\documentclass[12pt,twoside,reqno]{amsart}

\usepackage{amssymb,amsmath,amstext,amsthm,amsfonts}

\usepackage{dsfont}
\usepackage{color}

\usepackage[ansinew]{inputenc} 
\usepackage{graphicx}
\usepackage[mathscr]{eucal}
\usepackage{hyperref}

\newcommand{\R}{\mathbb{R}}
\newcommand{\C}{\mathbb{C}}
\newcommand{\N}{\mathbb{N}}

\newcommand{\G}{\mathbb{G}}

\newcommand{\Sc}{\overline{S}}
\newcommand{\Bscr}{\mathscr{B}}

\newcommand{\taboo}[2]{{}_{#1} p_{{#2}}}

\newcommand{\Rbar}{\overline{R}}

\newcommand{\us}[1]{u_{S'}^{(#1)}}

\newcommand{\abs}[1]{\bigl| #1 \bigr|} 
\newcommand{\norm}[1]{\lVert#1\rVert} 
\newcommand{\normtwo}[1]{
{\left\vert\kern-0.25ex\left\vert\kern-0.25ex\left\vert #1 
    \right\vert\kern-0.25ex\right\vert\kern-0.25ex\right\vert} }

\newcounter{main}

\numberwithin{equation}{section}

\newtheorem{theorem}{Theorem}[section]
\newtheorem{proposition}[theorem]{Proposition}
\newtheorem{lemma}[theorem]{Lemma}

\newtheorem{definition}{Definition}[section]
\newtheorem{maintheorem}{Theorem}

\newcommand{\blanksquare}{\,\,\,$\sqcup\!\!\!\!\sqcap$}

{\par}

\newcommand{\Pp}{\mathbb{P}}

\title[Eigenvectors of isospectral graph transformations]{\textbf{Eigenvectors of isospectral graph transformations}}

\date{\today}

\author[P. Duarte]{Pedro Duarte}
\address{CMAF, Departamento de Matem\'{a}tica,
  Faculdade de Ci\^{e}ncias da Universidade de Lisboa, Campo Grande, 1749-016 Lisboa, Portugal.}
\email{pedromiguel.duarte@gmail.com}

\author[M. J. Torres]{Maria Joana Torres}
\address{CMAT, Departamento de Matem\'atica e Aplica\c{c}\~{o}es, 
Universidade do Minho, 
Campus de Gualtar, 
4700-057 Braga, Portugal}
\email{jtorres@math.uminho.pt}

\begin{document}

\maketitle

\begin{abstract}
L.A. Bunimovich and  B.Z. Webb developed a theory for isospectral 
graph reduction.
We make a simple observation regarding the relation between eigenvectors of the original graph and its reduction,
that sheds new light on this theory.
As an application we propose an updating algorithm for the maximal eigenvector of the Markov matrix associated to
a large sparse dynamical network.
\end{abstract}

\bigskip


\smallskip


\bigskip

\begin{section}{Introduction}

The science of real world networks, i.e. networks found in nature, society and technology, has been in recent years one of the most active areas of research. 
Frequently cited examples include biological networks, social and economical networks, neural networks, networks of information like citation networks and the Internet (\cite{BA, DM, NB, W}).
These networks are typically described by a large and often complex graph of interactions, whose nodes represent elements of the network 
and whose  edges determine the topology  of interactions between these elements.
 
In~\cite{BW,BW2,BW3} it was introduced a tool, the {\em isospectral graph transformation}, that provides a way of understanding the interplay between the topology of a network and its dynamics.  More precisely, the authors introduce a concept of transformation of a graph (either by reduction or expansion) with the key property of preserving part of the spectrum of the graph's adjacency matrix.
In order to not contradict the fundamental theorem of algebra,  isospectral graph transformations preserve the spectrum of the graph (in particular the number of eigenvalues) by permitting edges to be weighted by functions of a spectral parameter $\lambda$ (see ~\cite[Theorem 3.5.]{BW}).
These transformations allow  changing the topology of a network (modifying the interactions, reducing or increasing the number of nodes),
while maintaining properties related to the network's dynamics.
In~\cite{BW3,VW} the authors  relate the pseudospectrum of a graph 
to the pseudospectrum of its reduction.

The work~\cite{BW} contains also an application of this isospectral theory to a class of dynamical systems modeling dynamical networks. In particular, they provide a  
 sufficient condition for a dynamical network  having a globally attracting fixed point.

The aim of this paper is to show that 
isospectral graph reductions preserve the eigenvectors associated to the eigenvalues of the graph's weighted adjacency matrix.  
This result can also be obtained as a corollary of the relation between  the pseudospectrum of a graph 
with the pseudospectrum of its reduction (see~\cite[equation 5.6 on p. 137]{BW3},  ~ \cite[equation 15 on p. 157]{VW}).

In section 2 we state and prove our main result,  explaining how to reconstruct
an eigenvector of the graph's adjacency matrix  from an eigenvector of the reduced matrix.

In section 3 we give a probabilistic interpretation and an application of the main theorem to Markov Chains.

In section 4  we describe an updating algorithm for {\em page rank} type eigenvectors, i.e.,
eigenvectors which assign the relative importance of each node in a dynamical network. 
We compare the cost of this algorithm with the cost of approximating a page rank eigenvector
by iterations of the updated adjacency matrix.

In the last section (appendix) we prove a  bound used to estimate the cost
of the updating algorithm in the previous section.
\end{section}

\bigskip
\bigskip

\begin{section}{Eigenvectors of isospectral graph reductions}\label{igr}

Let $\G$ be the class of weighted directed graphs, with edge weights in the set $\C$. 
More precisely, a graph $G \in \G$ is an ordered triple $G=(V,E,\omega)$ where $V=\{1,\ldots,n\}$ is the {\em vertex set},  $E \subset V \times V$ is the set of {\em directed edges},  and $\omega:E \rightarrow \C$ is the {\em weight function}. 
Denote by $M_G=(\omega(i,j))_{i,j \in V}$ the {\em weighted adjacency matrix} of $G$,
with the convention that $\omega(i,j)=0$ whenever $(i,j)\notin E$. 

A {\em path} $\gamma=(i_0,\ldots, i_p)$ in the graph $G=(V,E,\omega)$ is an ordered sequence of distinct vertices 
$i_0,\ldots, i_p \in V$ such that $(i_\ell,i_{\ell+1}) \in E$ for $0 \leq \ell \leq p-1$.
The vertices $i_1,\ldots,i_{p-1} \in V$ of $\gamma$ are called the {\em interior vertices}. If $i_0=i_p$ then $\gamma$ is a {\em cycle}. A cycle  is called a {\em loop} if $p=1$ and $i_0=i_1$.
The length of a path $\gamma=(i_0,\ldots, i_p)$ is the integer $p$.
Note that there are no paths of length $0$ and that  every
edge $(i,j)\in E$ is a path of length $1$.

 If $S \subset V$ we will
 write $\overline{S}=V \backslash S$.
 
\bigskip

\begin{definition} $(\lambda$-Structural set $)$\label{structural} \,
Let $G=(V,E,\omega)$. 
Given $\lambda \in \C$, a nonempty vertex set  $S \subset V$ is a {\em $\lambda$-structural set} of $G$ if
 \begin{enumerate}
 \item[$\text{(i)}$] each cycle of $G$, that is not a loop, contains a vertex in $S$; and
  \item[$\text{(ii)}$] $\omega(i,i) \neq \lambda$ for each $i \in \overline{S}$.
 \end{enumerate}
\end{definition}

\bigskip

Given a $\lambda$-structural set $S$, let us call {\em branch of}  $(G,S)$ to a  path
$\beta=(i_0,i_1,\ldots, i_{p-1}, i_p)$ such that  $i_1,\ldots, i_{p-1}\in \Sc$ and  $i_0, i_p\in V$.
 We denote by $\Bscr=\Bscr_{G,S}$ the set of all branches of $(G,S)$.
 Given vertices $i,j\in V$, we denote by
$\Bscr_{i j}$  the set of all branches in $\Bscr$ that start in $i$ and end in $j$. For each branch $\beta=(i_0,i_1,\ldots, i_p)$ we define
the {\em weight of $\beta$} as follows:
\begin{equation}\label{weight}
\omega(\beta,\lambda):= \omega(i_0,i_1)\,\prod_{\ell=1}^{p-1} \frac{\omega(i_{\ell},i_{\ell+1})}{\lambda-\omega(i_{\ell},i_\ell)} \;.
\end{equation}

\bigskip
Given $i,j \in V$ set
\begin{equation}\label{reduced:matrix}
 R_{ij}(G,S,\lambda):= \sum_{\beta\in \Bscr_{ij}} \omega(\beta,\lambda)\;.
\end{equation}

\bigskip

\begin{definition}$($Reduced matrix$)$\label{igr} \, Given $G \in \G$ and a $\lambda$-structural set $S$, the reduced matrix $R_S(G,\lambda)$ is the $S  \times S$-matrix with entries $R_{ij}(G,S,\lambda)$, $i,j \in S$.
\end{definition}

\bigskip

The following theorem states that 
isospectral graph reduction  preserves the eigenvectors associated to eigenvalues of the graph's weighted adjacency matrix. 

\bigskip

\begin{maintheorem}\label{eigenvector}
Given $G \in \G$, let
$\lambda_0$ be an eigenvalue of $M_G$ and $u=(u_1,u_2,\ldots,u_n) \in \C^n$ be the corresponding eigenvector, $M_G \,u = \lambda_0 u$. Assume that $S=\{m+1,\ldots,n\}$ is a $\lambda_0$-structural set of $G$. Then $\lambda_0$ is also an eigenvalue of $R_S(G,\lambda_0)$ and $R_S(G,\lambda_0)\, u_S = \lambda_0 u_S$, where $u_S=(u_{m+1},\ldots,u_n)$ is the restriction of $u$ to $S$.
\end{maintheorem}

\bigskip

In the proof of Theorem~\ref{eigenvector} we will 
use the following useful notation. Given vertices $i,j\in S$, we denote by
$\Bscr^{(p)}_{i j}$ the set of all branches in $\Bscr$ of lengh $p$ that  start in $i$ and end in $j$. Given $i,j \in S$ set
$$ R^{(p)}_{ij}(G,S,\lambda):= \sum_{\beta\in \Bscr^{(p)}_{ij}} \omega(\beta,\lambda)\;.$$
With this notation, it is clear
that the reduced weights $R_{ij}(G,S,\lambda)$, $i,j \in S$, associated to the structural set $S=\{m+1,\ldots,n\}$ satisfy
$$R_{ij}(G,S,\lambda)=\sum_{p=1}^{m+1}  R^{(p)}_{ij}(G,S,\lambda).$$
To simplify the notation we will 
write $R_{ij}$ and $R^{(p)}_{ij}$ instead of $R_{ij}(G,S,\lambda)$ and $R^{(p)}_{ij}(G,S,\lambda)$, respectively.
We now prove Theorem~\ref{eigenvector}.

\begin{proof} 
Clearly, the eigenvector $u$ can be written as
$u=(u_{\overline{S}},u_S)$, where 
$u_{\overline{S}}=(u_\ell)_{\ell \in \overline{S}}$ and $u_S=(u_i)_{i \in S}$.
Since $M_G \,u = \lambda_0 \, u$, one has for all $\ell \in \overline{S}$,
\begin{equation*}
\sum_{k \in S} \omega(\ell,k)\, u_k \, + \,  \omega(\ell,\ell) \, u_\ell \, + \,  
\sum_{  \substack{ \ell' \in \overline{S} \\  \ell' \neq \ell  } } 
 \omega(\ell,\ell') \,u_{\ell'}  = \lambda_0 \, u_\ell\;,
\end{equation*}
which is equivalent to
\begin{equation}\label{iteration}
u_\ell=\sum_{k \in S} \frac{\omega(\ell, k)}{\lambda_0 - \omega(\ell, \ell)} \, u_k \, + \,  
 \sum_{ \substack{ \ell' \in \overline{S} \\   \ell' \neq \ell  } } 
\frac{\omega(\ell,\ell')}{\lambda_0 - \omega(\ell,\ell)} \, u_{\ell'}.  
\end{equation}
Hence for all $i \in S$, 
$$
u_i   = \displaystyle{\sum_{
{\substack{ 
k \in S \\
k \neq i 
}
}} \frac{\omega(i,k)}{\lambda_0 - \omega(i,i)} \, u_k 
\, + \,  
\sum_{
\ell \in \overline{S} } 
\frac{\omega(i,\ell)}{\lambda_0 - \omega(i,i)} \, u_{\ell}} \;. $$
Substituting $u_\ell$ by ~(\ref{iteration}) in this relation we get

\begin{align*} 
u_i & =  \sum_{ \substack{ k \in S \\ k \neq i } } 
\frac{\omega(i,k)}{\lambda_0 - \omega(i,i)} \, u_k 
\, + \,
 \sum_{ \substack{ k \in S \\  \ell \in \overline{S} } } 
 \frac{\omega(i,\ell) \, \omega(\ell, k)}{(\lambda_0 - \omega(i,i))(\lambda_0 - \omega(\ell,\ell))} \, u_k   \\
 & \qquad  + \,  \sum_{ \substack{ \ell, \ell' \in \overline{S} \\ \ell' \neq \ell } } 
 \frac{\omega(i,\ell) \, \omega(\ell,\ell')}{(\lambda_0 - \omega(i,i))(\lambda_0 - \omega(\ell,\ell))} \, u_{\ell'}  \\ 
  &   =   \sum_{ \substack{ k \in S \\ k \neq i } }  
  \frac{R^{(1)}_{i k}}{\lambda_0 - \omega(i,i)} \, u_k 
\, + \,
\sum_{k \in S} \frac{R^{(2)}_{i k}}{ \lambda_0 - \omega(i,i) } \, u_k 
\, + \,
\sum_{ \substack{ \ell, \ell' \in \overline{S} \\ \ell' \neq \ell } } 
\frac{\omega(i,\ell) \, \omega(\ell,\ell')}{(\lambda_0 - \omega(i,i))(\lambda_0 - \omega(\ell,\ell))} \, u_{\ell'}. 
\end{align*}

\medskip
\noindent
Proceeding inductively,  we obtain for all $1 \leq p \leq m$

\begin{align*}
\begin{split}
u_i=  &  \sum_{\substack{ k \in S \\ k \neq i } } 
\frac{R^{(1)}_{i k}}{\lambda_0 - \omega(i,i)} \, u_k 
\, + \,
\sum_{k \in S} \frac{R^{(2)}_{i k}}{ \lambda_0 - \omega(i,i) } \, u_k 
\, + \,
\cdots
\, + \,
\sum_{k \in S} \frac{R^{(p)}_{i k}}{ \lambda_0 - \omega(i,i) } \, u_k
\,  +  \\
& 
+ \,
\sum_{
\substack{  \ell_1,\cdots,\ell_{p-1}, \ell' \in \overline{S} \\
\ell_r \neq \ell_s; \, \ell_s \neq \ell } } 
\frac{\omega(i,\ell_1) \,\omega(\ell_1,\ell_2)\,\cdots\, \omega(\ell_{p-1}, \ell')}{(\lambda_0 - \omega(i,i))(\lambda_0 - \omega(\ell_1,\ell_1)) (\lambda_0 - \omega(\ell_2,\ell_2)) \cdots (\lambda_0 - \omega(\ell_{p-1},\ell_{p-1}))  } \, u_{\ell'}.
\end{split}
\end{align*}
\bigskip
\noindent
Note that the indices in $\overline{S}$ are all distinct because there are no non-loop cycles  in $\overline{S}$.
Since $\overline{S}=\{1,\ldots,m\}$ has $m$ elements, after $m+1$ steps we get that
$$u_i= \sum_{\substack{ k \in S \\ k \neq i } } 
\frac{R^{(1)}_{i k}}{\lambda_0 - \omega(i,i)} \, u_k 
\, + \,
\sum_{k \in S} \frac{R^{(2)}_{i k}}{ \lambda_0 - \omega(i,i) } \, u_k 
\, + \,
\cdots
\, + \,
\sum_{k \in S} \frac{R^{(m+1)}_{i k}}{ \lambda_0 - \omega(i,i) } \, u_k\; ,
$$
which is equivalent to
$$ (\omega(i,i) - \lambda_0)\,u_i 
\, +\,
\sum_{p=2}^{m+1} R^{(p)}_{ii} \, u_i
\, + \,
\sum_{\substack{ k \in S \\ k \neq i } }
\sum_{p=1}^{m+1}R^{(p)}_{i k} \,u_k =0\;.
$$
This in turn  is equivalent to
$$ \left(\sum_{p=1}^{m+1} R^{(p)}_{ii} - \lambda_0\right)  u_i
\, + \,
\sum_{\substack{ k \in S \\ k \neq i } }
\sum_{p=1}^{m+1}R^{(p)}_{i k} \, u_k =0\;. 
$$
Therefore, we get that
$$(R_{ii} - \lambda_0) \, u_i \, + \, \sum_{
{\tiny \begin{array}{c}
k \in S \\
k \neq i
\end{array}
}}
R_{ik} \, u_k =0\;,
$$
and hence for all $i \in S$, 
$$ \sum_{
k \in S
}
R_{ik} \, u_k = \lambda_0 \,u_i\;.
$$
This proves that
$\lambda_0$ is also an eigenvalue of $R_S(G,\lambda_0)$ with
$R_S(G,\lambda_0)\, u_S = \lambda_0 u_S$.
\end{proof}

\bigskip

To explain how to reconstruct the eigenvector of
$M_G$ from the eigenvector of the reduced matrix $R_S(G,\lambda_0)$ we need the following concept of depth of vertex $i\in V$.

\bigskip

\begin{definition}$($Depth of a vertex$)$\label{depth} \,
The  depth of a vertex  $i\in V$ is defined recursively as follows.
\begin{enumerate}
\item A vertex $i\in S$ has depth $0$.
\item A vertex $i\in\Sc$ has depth $k$ iff\,  
$i$ has no depth less than $k$, and 
$(i,j)\in E$ implies $j$ has depth $<k$, for all $j\in V$.
\end{enumerate}
We denote by $S_k$ the set of all vertices of depth $\leq k$.
Because $S$ is a  structural set, every vertex $i$ has a finite depth. We call  depth of $(G,S)$ to the maximum depth of a vertex.
\end{definition}

Figure~\ref{nivel} shows the depth hierarchy of a graph $G$ with vertex set $V=\{1,2,3,4,5,7\}$ and structural set $S=\{1,3,5\}$.

\begin{figure}[h]
\includegraphics[width=0.45\textwidth]{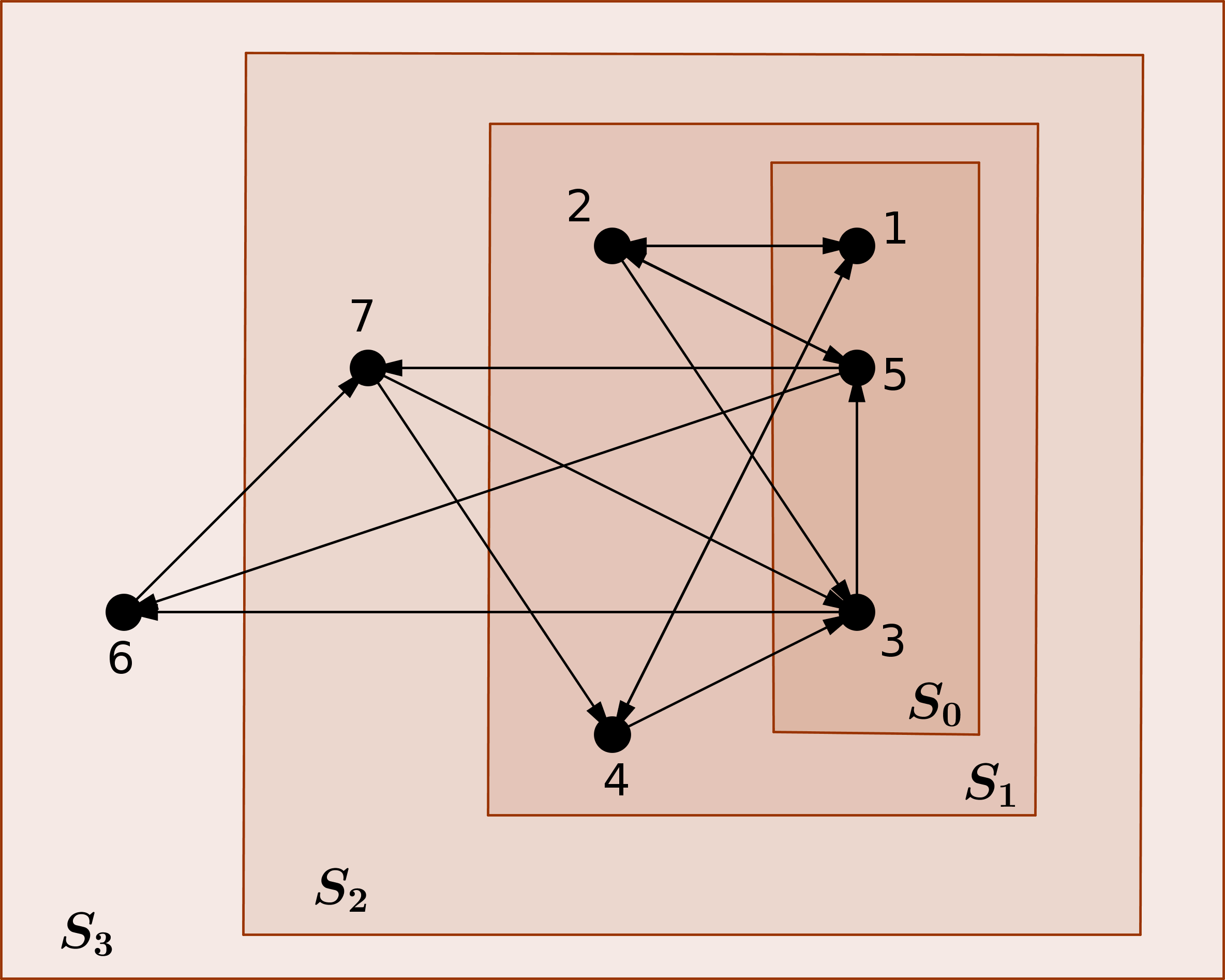}
\caption{Depth hierarchy of a graph}
\label{nivel}
\end{figure}

If $\lambda_0$ is an eigenvalue of $M_G$,
by Theorem~ \ref{eigenvector} it is also an eigenvalue of the reduced matrix $R_S(G,\lambda_0)$.
Knowing the eigenvector $u_S$ of this reduced matrix,
we can recover the corresponding eigenvector of $M_G$  as follows:

\bigskip

\begin{proposition}
If $\lambda_0$ is an eigenvalue of $M_G$  and $u_S=(u^{S}_{i})_{i\in S}$ is an eingenvector of the reduced matrix
$R_S(G,\lambda_0)$  then the following recursive relations 
\begin{equation}\label{recurs}
\left\{\begin{array}{l}
u_i  = u^ {S}_{i}  \quad \text{ for } \;  i\in S_0=S\\
\smallskip\
{\displaystyle u_\ell  = \sum_{j\in S_{k-1}} {\frac{\omega(\ell,j)}{\lambda_0-\omega(\ell,\ell)}\,u_j} } 
\quad \text{ for all }\;
\ell\in S_k\setminus S_{k-1} 
\end{array}\right.
\end{equation}
uniquely determine an eigenvector $u$ of $M_G$ associated to   $\lambda_0$.
\end{proposition}

\begin{proof}
Just notice that if $\ell\in S_k\setminus S_{k-1} $ then the second part of ~\eqref{iteration}
can only contain terms $\frac{\omega(\ell,\ell')}{\lambda_0 - \omega(\ell,\ell)} \, u_{\ell'}$ with 
$\ell'\in S_{k-1}$. Hence~  \eqref{recurs} follows from ~ \eqref{iteration}.
\end{proof}

\end{section}

\bigskip
\bigskip

\begin{section}{A probabilistic interpretation}\label{interpretations}

In this section we present a probabilistic  interpretation of Theorem 1.

\bigskip

Consider a graph $G\in\G$  such that $M_G=(\omega(i,j))_{i,j}$ is a stochastic matrix.
More precisely assume $\omega(i,j)=p_{ij}$ is the transition probability
of some finite state Markov Chain (MC)  $\{X_n\}_{n\in\N}$ with state space $V$. Note that in this case  $\lambda=1$ is an eigenvalue of $M_G$.

Assume that  $S\subset V$ is a $1$-structural set 
such that $G$ has no loops in 
$\overline{S}$, 
i.e., $p_{ii}=0$ for all $i\notin S$.
We  consider the reduced matrix $R_S(G,\lambda)$ with    $\lambda=1$, denoted hereafter by  $R_S(G)$.

We will interpret the entries of $R_S(G)$ as taboo probabilities.
Recall that given a set $S\subset V$, the {\em taboo probability } 
$\taboo{S}{ij}^{(n)}$ is defined as (see ~\cite{chung})
$$ \taboo{S}{ij}^{(n)} := \Pp[\,X_n=j,\,
X_k\notin S, \, 0<k<n\, \vert\, X_0=i\, ]\;. $$
If $i,j\in S$, this is the probability of the MC returning for the first time to $S$, with  state $j$,
at time $n$, 
given that the process starts in state $i$ at time $0$.

Because
 $\lambda=1$ and $\omega(\ell,\ell)=p_{\ell\ell}=0$
for any $\ell \in \overline{S}$, applying Bayes' theorem we derive from
~\eqref{weight} that
$$ R_{ij}^{(n)}(G,S) =  \taboo{S}{ij}^{(n)} \;. $$
Hence
$$ R_{ij}(G,S) =  \sum_{n=1}^\infty \taboo{S}{ij}^{(n)} $$ 
 is the probability of the process returning for the first time to $S$, with  state $j$, given that it starts in state $i$ at time $0$.

We remark that the (normalized) eigenvectors of $M_G$, corresponding to the eigenvalue $\lambda=1$, are precisely the stationary distributions of the given Markov process.

Define recursively the sequence of stopping times 
\begin{align*}
\tau_0(\omega) &= \min \{\, k\geq 0\, \colon \, X_k(\omega)\in S\,\} \\
\tau_n(\omega) &= \min \{\, k\geq \tau_{n-1}(\omega)+1 \, \colon \, X_k(\omega)\in S\,\} \;,
\end{align*}
and consider the  $S$-valued random process $Y_n(\omega):= X_{\tau_n(\omega)}(\omega)$. The previous considerations show that this process is a MC with transition probability matrix $R_S(G)$.

Thus, from Theorem 1 we derive the following result in probability theory.

\bigskip

\begin{proposition}
Let $X_n$ be a finite MC with state space $V$ and consider the  graph
$G\in\G$ determined by the MC's transition probabilities $p_{ij}$. Assume 
\begin{enumerate}
\item $S\subset V$ is a $1$-structural set 
of $G$
such that $p_{ii}=0$ for all $i\in \overline{S}$,
\item $q=(q_i)_{i\in V}$ is a stationary distribution of the Markov chain $X_n$.
\end{enumerate}
Then the  process $Y_n=X_{\tau_n(\omega)}$ is a stationary MC with transition probability matrix 
$R_S(G)$ and stationary distribution $q_S=(q_j/\sum_{k\in S} q_k)_{j\in S}$.
\end{proposition}

\end{section}

\bigskip
\bigskip

\begin{section}{An application: page rank type eigenvectors}

The motivation for the following application of Theorem~\ref{eigenvector} was the {\em PageRank} algorithm used by the Google search engine to calculate a web page's importance.
The {\em PageRank} algorithm assigns the relative importance of each web page by computing the dominant eigenvector (the {\em page rank vector}) of
a particular stochastic weighted adjacency matrix of the World Wide Web graph (see~\cite{BP}). 

In general, given a network we can assign the importance of each node by computing the dominant eigenvector of a weighted adjacency matrix of the correspondent graph of interactions.
Thus, our goal in this application is the following. 
Suppose that we have a dynamical network (like the World Wide Web) that changes its topology (modifying its interactions, reducing or increasing the number of nodes),
and that we want 
to update the dominant eigenvector of the modified network. We show that (in some cases) it is computationally more efficient
to perform an isospectral graph reduction, calculate the dominant eigenvector of the reduced graph and then use this vector to update the dominant eigenvector  of the entire modified network rather than use the recalculated matrix of the whole modified network to compute the dominant eigenvector through an iterative procedure.

\bigskip

Let $G=(V,E,\omega)$ be a weighted directed graph where $V=\{1,2,\ldots, N\}$. 
We assume that
\begin{enumerate}
\item $(i,i)\notin E$, for all $i\in V$,
\item $\omega(i,j) \in [0,1]$ for all $i,j\in V$, 
\item $\omega(i,j) >0$ \, $\Leftrightarrow$\, $(i,j)\in E$, for all $i,j\in V$,
\item $\sum_{i=1}^N \omega(i,j)=1$ for all $j \in V$, i.e., $M_G$ is a stochastic matrix,
\item $S\subset V$ is a $1$-structural set (in the sense of Definition~\ref{structural}),
\item $M_G$
is a stochastic primitive matrix.
\end{enumerate}
The matrix $M_G$ should be sparse in some sense for the described goal to be attainable.
In particular, by the Perron-Frobenius theorem,  $M_G$ has a unique dominant  eigenvector
$v_G=(v_i)_{i\in V}\in\R^N_+$, normalized by the condition
$\sum_{j\in V} v_j=1$, associated to the eigenvalue $1$. We will refer to $v_G$ as the {\em dominant eigenvector} of the primitive matrix $M_G$.

\bigskip

Recall that $\Bscr=\Bscr_{G,S}$ denotes the set of all branches of $(G,S)$.
Given vertices $i,j\in V$, we denote respectively by
$\Bscr_{i j}$, $\Bscr_{i \ast}$, $\Bscr_{\ast j}$ and $\Bscr_{\ast i \ast}$, the sets of all branches in $\Bscr$ that start in $i$ and end in $j$, respectively that start in $i$, that  end in $j$, and that  go through $i$. 
Notice that, since $\lambda=1$ and the graph has no loops,
for each branch 
$\beta=(i_0,i_1,\ldots, i_p)$
we have (see~\eqref{weight})
$$ \omega(\beta) = \prod_{\ell=1}^p \omega(i_{\ell-1},i_\ell) \;. $$
Recalling ~\eqref{reduced:matrix}, we have for all  $i,j \in V$ 
$$ R_{ij}(G,S) = \sum_{\beta\in \Bscr_{ij}} \omega(\beta)\;.$$
In Definition~\ref{igr} we have introduced the reduced matrix
$R_S(G)$, that we now extend according to the following definition.

\bigskip

\begin{definition}\label{extended:igr}
We call {\em extended reduced matrix} to the $N\times N$ matrix $\Rbar_S(G)$
with entries $R_{ij}(G,S)$. Let us denote by $M_{\Sc}$ the matrix $\left( \omega(i,j)\right)_{i,j\in \Sc}$, i.e., the restriction of $M_G$ to $\Sc$.
\end{definition}

\bigskip

\begin{proposition}
The set $S$ is a  structural set of $G$ if and only if 
$M_{\Sc}$ is nilpotent. Moreover,  the depth of $(G,S)$ is the smallest $k\in\N$ such that
$(M_{\Sc})^k=0$.
\end{proposition}

\begin{proof}
We denote by $D_k$ the set of all vertices in $\Sc$ of depth $k$. We just need to observe that there exists a permutation matrix $P$ such that
$$P\,M_{\Sc}\,P^{-1}= \left( 
\begin{array}{ccccc}
0 & N_{k} & * & \hdots & * \\ \
0 & 0 & N_{k-1} & \hdots & * \\ \
0 & 0 & 0 & \ddots & * \\ \
0 & 0 & 0 & \hdots & N_2 \\ \
0 & 0 & 0 & \hdots & 0 \\ \
\end{array}
  \right)$$
  where the matrices $N_i$, $2 \leq i \leq k$ are indexed by vertices in $D_{i} \times D_{i-1}$, respectively. 
\end{proof}

\bigskip
\bigskip

In the sequel we care about the following measurements
\begin{itemize}
\item $N=\abs{V}$ is graph's number of vertices,
\item $s=\abs{S}$ is the cardinal of the structural set,
\item $k$ \, is the depth of $(G,S)$,
\item $m=\max\{ \max_{i\in \Sc} \abs{\Bscr_{\ast i \ast}},
\max_{i\in V} \abs{\Bscr_{  \ast i}},
\max_{i\in V} \abs{\Bscr_{ i \ast}}
\}$  is the maximum number of branches through any vertex $i$,
\item $\ell$ is the maximum number of iterations
allowed to approximate the eigenvectors,
\item $p$ is the maximum number of vertices or edges
to be added to the graph,
\end{itemize}
which we assume to satisfy the relations
\begin{equation}\label{meas}
 p\ll s \ll N, \quad    k+p \ll N\quad \text{ and }
\quad  p\cdot (k+1)\cdot m \ll N^3\;.
\end{equation}

\bigskip
\bigskip
The following data is stored, in the proposed algorithm:
\begin{itemize}
\item the $N\times N$ matrix $M_G$, 
\item the structural set $S$,
\item the set $\Bscr$ of all branches of $(G,S)$,
\item the  extended reduced $N\times N$ matrix $\Rbar_S(G)$,
\item the normalized dominant eigenvector of the reduced matrix $R_S(G)$,
\item the normalized dominant eigenvector of matrix $M_G$.
\end{itemize}

\bigskip

Besides the stored data, the input of the algorithm will consist on a short list of vertices and edges to be added to, or removed from, the original graph. Next we  briefly describe the main steps of an {\em Updating Algoritm}
to recompute the stored data for the modified graph,
denoted henceforth by $(G',S')$.

\begin{enumerate}
\item Update the matrix  $M_{G'}$.
\item  Check if the structural set $S$ remains structural for $G'$, and recompute $S'$ if necessary, by adding some 
of the new edges' endpoints.

\item Update the set $\Bscr'$ of branches of $(G',S')$.

\item Update the extended reduced $N\times N$ matrix $\Rbar_S(G')$.

\item Recompute the dominant eigenvector of the reduced matrix $R_{S'}(G')$.
\item Update the dominant eigenvector of matrix $M_{G'}$.

\end{enumerate}

\bigskip

Next we give some rough estimates on the corresponding computational costs.

\begin{enumerate}
\item The cost of updating $A_{G'}$ and $M_{G'}$ is comparatively very small, and will
be neglected. Note that the number of modified entries of these matrices is
$p\ll s \ll N$.

\item New vertices do not change the structural set.
Let us compute the cost of  updating $S$ by adding a new edge $(i,j)$.
If $i,j\in\Sc$ and $R_{ji}(G,S)>0$, then 
$S$ is no longer a structural set of the new graph $G'=(V,E')$, 
where $E'=E\cup \{(i,j)\}$. 
In this case  we set $S':=S\cup \{i\}$. Otherwise $S$ is still 
a structural set for $G'$, and we set $S':=S$.
Since the number of new edges is small this cost is  comparatively low, and will 
be neglected.

\item We consider first the  cost of updating $\Bscr'$ by adding a new edge $(i,j)$. We divide this cost in two cases:
	\begin{enumerate}
	\item[(i)] Assume  $S'=S$. In this case 
	all existing branches remain. The set of new branches can be identified with $\Bscr_{ \ast i}\times\Bscr_{j \ast}$, and is contained in 
	$\Bscr_{\ast i \ast}\cap \Bscr_{\ast j\ast}$. Therefore,
	since all branches have at most length $k$, the updating 
	cost is of order $(k+1)\cdot m$.

	\item[(ii)] Assume $S'=S\cup\{i\}$. In this case there are no new branches, but we have to delete all branches through $i$, i.e., branches in $\Bscr_{\ast i \ast}$.
	Therefore, the updating cost is of order $k\cdot m$.
	\end{enumerate}
The general cost of updating $\Bscr'$ by adding or deleting up to $p$ objects (vertices or edges) is at most $p\cdot (k+1)\cdot  m$.

\item For the cost of updating the extended reduced $N\times N$ matrix  $\Rbar_{S'}(G')$, note that for each deleted, respectively added, branch $\beta$
 we have to subtract from, respectively add to, 
$\Rbar_S(G)$ the entry $\omega(\beta)$ which is a product of at 
most $k+1$ factors.
Hence, since there are at most $p\cdot m$ modified branches,  the total cost is at most $p\cdot (k+1)\cdot m$.

\item To approximate the dominant eigenvector of  the $s\times s$ reduced matrix $R_{S'}(G')$,
we have to iterate this matrix  $\ell$ times.
The corresponding cost is $\ell\, s^3$.

\item Finally, to compute the dominant eigenvector of  $M_{G'}$,
we have to use the recursive relations~\eqref{recurs}, with the following computational cost 
$$ \sum_{j=1}^{k'}  j\, \abs{S_{j-1}'}\, \abs{S_j'\setminus S_{j-1}'} =
 \sum_{j=1}^{k'} j\, \abs{S_{j-1}'}\, (\abs{S_j'}-\abs{S_{j-1}'}  ) \;, $$
 where $k'$ is the depth of $G'$, and  $S'_j$ is the set of vertices of  $(G',S')$ with depth $\leq j$. 
 By Lemma~\ref{comb} (in the appendix) and since $k'\leq k+p$, this cost is at most $(k+p)\cdot N^2/2$.
\end{enumerate}

\bigskip

The computational cost for updating the dominant eigenvector 
of $M_{G'}$ with $\ell$ iterations of this matrix is $\ell \, N^3$.
Adding up  the partial costs above, associated with the proposed {\em updating algorithm}, 
and minding ~\eqref{meas}, we see that all partial updating costs  (3)-(6)  are clear improvements on the global cost  $\ell \, N^3$.

The same procedure can be applied to non-stochastic
primitive matrices satisfying assumptions (1)-(4)
by keeping track of both the eigenvalue $\lambda$ and the eigenvector $u_S$ of
the reduced matrix $R_{S'}(G',\lambda)$ in the following iterative procedure:

\begin{align*}
\us{n} &= {\rm versor}\left[  R_{S'}(G',\lambda_{n-1})\, \us{n-1}\right] \\
\lambda_n &= \frac{\norm{R_{S'}(G',\lambda_{n-1})\, \us{n-1}}}{\norm{\us{n-1}}}.
\end{align*}

\bigskip

The World Wide Web graph has over $N=25\times 10^9$
and average degree of about  $10$  links per page (see~\cite{austin}).
These numbers, from 2006, are surely out-of-date.
We just mention them to stress how sparse is the Google matrix.

We have used Wolfram Mathematica to randomly generate some 
sparse graphs, and compare the cost of updating the dominant eigenvector according to the algorithm above with corresponding global cost of updating it through an iterative procedure.

\begin{figure}[h]
\includegraphics[width=0.8\textwidth]{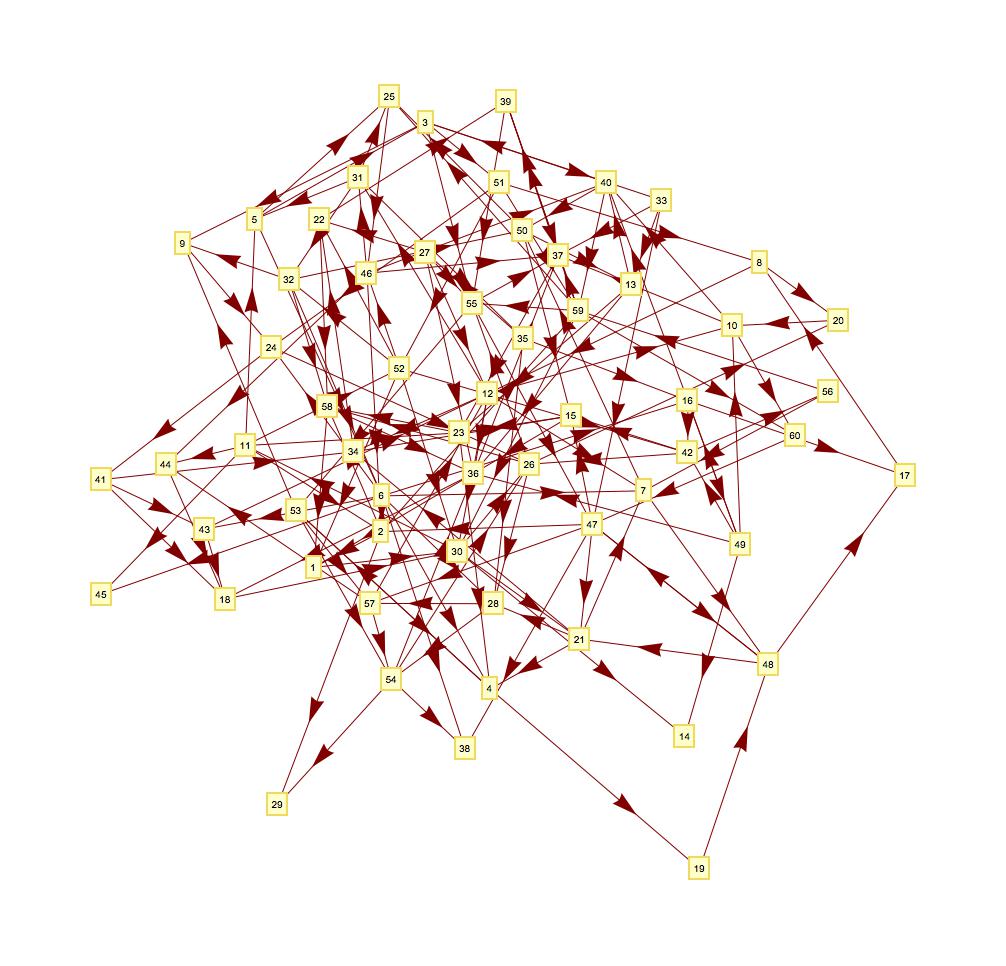}
\caption{A randomly generated graph with $N=60$}
\label{randgraph}
\end{figure}

We have considered graphs with no more than  $N\leq 60$
nodes with an average degree of $2$ or $3$ links per node.
For each generated graph we have computed a structural set $S$,
its depth $k$, and the maximum number $m$ of branches through any node. We have always considered  $p=3$ vertex 
or edge modifications in the graph.
The  standard updating of the dominant eigenvector was considered to take $\ell=10$ iterations of the graph's matrix.
Note that we didn't implement the updating algorithm but only compared the costs made explicit above.

An empirical (expected) observation is that the size and depth of the structural sets decrease  when the random graphs become sparser. Under the previous specifications,
most of the graphs randomly generated  satisfied the conditions \eqref{meas}, and the computational cost savings of the updating algorithm relative to the standard iterative algorithm were most of the times over $70\%$.

Figure~\ref{randgraph}  shows a randomly generated graph with $N=60$,
 $s=14$, $k=13$ and  $m=1125$.
In this case, based on the derived bounds, the cost savings amount to
$87.94\%$.

\end{section}

\bigskip

\begin{section}{Appendix}

We  prove here an auxiliary lemma  used to estimate the computational cost of updating the dominant eigenvalue of matrix $M_{G'}$.

Given $m\in\N$ and $N\in\R$, define the $m$-dimensional
simplex
$$ \Delta^ m_N=\{\,(x_0,x_1,\ldots, x_m)\,:\,
0\leq x_0\leq x_1\leq \ldots \leq x_m=N\,\}\;,$$
and the function $F:\Delta^ m_N\to\R$,
$$ F(x_0,x_1,\ldots, x_m):= \sum_{i=1}^m x_{i-1}\,(x_i-x_{i-1})\;. $$

\begin{lemma}\label{comb}
For all $(x_0,x_1,\ldots, x_m)\in\Delta^m_N$,
$$ \sum_{i=1}^m x_{i-1}\,(x_i-x_{i-1})\leq \frac{m\,N^2}{2\,(m+1)} \leq \frac{N^2}{2}\;. $$
\end{lemma}

\begin{proof}
The proof goes by induction in $m$.
For $m=1$ this follows from the fact that the function
$F:[0,N]\to\R$, $F(x)=x\,(1-x)$ attains its maximum
value, $N^2/4$, at $x= N/2$.
Assume the inequality holds for $m-1$.
At the point $p=(\frac{N}{m+1}, \frac{2\,N}{m+1},\ldots,
\frac{(m+1)\,N}{m+1})\in\Delta^m_N$ the value is
$F(p)= \frac{m\,N^2}{2\,(m+1)}$. Given $x\in\partial\Delta^m_N$, then either $x_0=0$ or else $x_{j}=x_{j+1}$ for some $j=0,1,\ldots, m-1$. In any case,
dropping the coordinate $x_j$ we get a sequence
$x'\in\Delta^{m-1}_N$ with the same coordinates as $x$.
By definition of $F$, and the induction hypothesis,
$F(x)=F(x')\leq \frac{(m-1)\,N^2}{2\,m}<\frac{m\,N^2}{2\,(m+1)}$, which proves that boundary points of
$\Delta^m_N$ can never be maxima. Thus  $F$ has an interior maximum. Computing the gradient of $F$, all  critical points $(x_0,x_1,\ldots, x_m)\in {\rm int}(\Delta^m_N)$ satisfy for all $j=1,\ldots, m-1$,
$$ x_{j-1}+x_{j+1}-2\,x_j =0\;. $$
In other words, the coordinates of $x$ form an arithmetic progression.
This shows $p$ is the only
critical point of $F$ in ${\rm int}(\Delta^m_N)$, 
 and  hence the claimed inequality  must hold for $m$.
\end{proof}

\end{section}

\bigskip

\section*{Acknowledgements}
The first author was partially supported by  Funda\c{c}\~{a}o para a Ci\^{e}ncia e a 
Tecnologia, PEst, OE/MAT/UI0209/2011.

The second author was partially supported by the Research Centre of Mathematics of the University of Minho with the Portuguese Funds from the ``Funda\c c\~ao para a Ci\^encia e a Tecnologia", through the Project PEstOE/MAT/UI0013/2014.

\bigskip


\bigskip


\begin{thebibliography}{ABC}


\bibitem{BA} {R. Albert and A-L Barab\'{a}si}, \emph{Statistical mechanics of complex networks}, Rev. Mod. Phys. {\bf 74} (2002), 47-97.

\bibitem{austin} {D. Austin}, \emph{How Google Finds Your Needle in the Web's Haystack}, December 2006, [AMS Feature Column-Monthly essays on mathematical topics; posted December-2006].

\bibitem{BP} {S. Brin and L. Page}, \emph{The anatomy of a large-scale hypertextual Web search engine}, Computer Networks and ISDN Systems {\bf 30} (1998), 107-117.

\bibitem{BW2} {L. A. Bunimovich  and B. Z. Webb}, \emph{Isospectral graph reductions and improved estimates of matrices' spectra}, Linear Algebra and its Applications {\bf 437} (2012), 1429-1457.

\bibitem{BW} {L. A. Bunimovich  and B. Z. Webb}, \emph{Isospectral graph transformations, spectral equivalence, and global stability of dynamical networks}, Nonlinearity {\bf 25} (2012), 211-254.

\bibitem{BW3} {L. A. Bunimovich  and B. Z. Webb}, \emph{Isospectral transformations}, Springer Monographs in Mathematics, Springer, New York, 2014.


\bibitem{chung} {K. L. Chung}, \emph{A course in probability theory}, Second Edition, Academic Press, New York, 1974.

\bibitem{DM} {S. N.  Dorogovtsev and J. F. F. Mendes}, \emph{Evolution of Networks:
From Biological Nets to the Internet and WWW}, Oxford University  Press, 2003.

\bibitem{NB} {M. Newman, A-L Barab\'{a}si, and D. J. Watts}, \emph{The Structure and Dynamics of Networks}, Princeton University Press,  2006.

\bibitem{VW} {F. G. Vasquez and B. Z. Webb}, \emph{Pseudospectra of isospectrally reduced matrices}, Numerical Linear Algebra with Applications {\bf 22} (2015), 145-174.

\bibitem{W} {D. J. Watts}, \emph{Small Worlds: The Dynamics of Networks between Order and Randomness}, Princeton University Press, 1999.


\end{thebibliography}
\end{document}